\documentclass{amsproc}

\usepackage{amsmath}
\usepackage{amssymb}
\usepackage{graphicx}
\usepackage{amsthm}
\usepackage[all,2cell]{xy}

\newtheorem{thm}{Theorem}[section]
\newtheorem{lem}[thm]{Lemma}
\newtheorem{cor}[thm]{Corollary}
\newtheorem{prop}[thm]{Proposition}

\newtheorem{claim}[thm]{Claim}

\newtheorem{defn}[thm]{Definition}

\newtheorem{rem}[thm]{Remark}

\newtheorem{ex}[thm]{Example}

\def\C{{\mathbb C}}
\def\F{{\mathbb F}}
\def\G{{\mathbb G}}

\def\O{{\mathbb O}}
\def\P{{\mathbb P}}
\def\Q{{\mathbb Q}}
\def\R{{\mathbb R}}

\def\cF{{\mathcal F}}

\def\cL{{\mathcal L}}
\def\cM{{\mathcal M}}

\def\cO{{\mathcal{O}}}
\def\cQ{{\mathcal{Q}}}

\def\cS{{\mathcal S}}

\def\cU{{\mathcal U}}

\def\operatorname#1{\mathop{\rm #1}\nolimits}

\def\Loc{\operatorname{Locus}}
\def\Proj{\operatorname{Proj}}

\def\Pic{\operatorname{Pic}}

\def\rk{\operatorname{rk}}

\def\deg{\operatorname{deg}}

\def\det{\operatorname{det}}

\def\ratcurves{\operatorname{RatCurves}}

\def\Sl{\operatorname{Sl}}
\def\SO{\operatorname{SO}}
\def\Sp{\operatorname{Sp}}

\def\NE{\operatorname{NE}}

\begin{document}

\title[Uniform vector bundles on Fano manifolds]{Uniform vector
bundles on Fano manifolds and applications}

\author{Roberto Mu\~noz}
\address{Departamento de Matem\'atica Aplicada, ESCET, Universidad Rey
  Juan Carlos, 28933-M\'ostoles, Madrid, Spain}
%\curraddr{Department of Mathematics and Statistics,
%Case Western Reserve University, Cleveland, Ohio 43403}
\email{roberto.munoz@urjc.es}
\thanks{First and third authors partially supported by the Spanish government
project
  MTM2006-04785. Second author partially supported by  MIUR (PRIN project:
Propriet\`a
geometriche delle variet\`a reali e complesse)}

\author{Gianluca Occhetta} \address{Dipartimento di Matematica,
  Universit\`a di Trento, Via Sommarive 14,
I-38123, Povo (Trento), Italy}
\email{gianluca.occhetta@unitn.it}
%\thanks{Partially supported by }

%    Information for third author
\author{Luis E. Sol\'a Conde} \address{Departamento de Matem\'atica
  Aplicada, ESCET, Universidad Rey Juan Carlos, 28933-M\'ostoles,
  Madrid, Spain} \email{luis.sola@urjc.es}
%\thanks{Partially supported by the Spanish government project MTM2006-04785.}

%    General info
%\date{}


\begin{abstract}
In this paper we give a splitting criterion for uniform vector bundles on Fano
manifolds covered by lines. As a consequence, we classify low rank uniform
vector bundles on Hermitian symmetric spaces and Fano bundles of rank
two on Grassmannians.
\end{abstract}

\maketitle
\tableofcontents

\section{Introduction}\label{sec:intro}

It is classically known that every vector bundle on the projective line splits 
as a sum of line bundles, and hence it is determined, up to 
isomorphism, by a list of integers, usually called its {\it splitting type}. 
It is then natural to consider restrictions to lines in order to study vector bundles on projective spaces. Although the situation for spaces of 
dimension bigger than or equal to two is much more involved, one may obtain 
partial classification results after restricting to certain classes of vector 
bundles. One of the classes that has been studied more extensively is that 
of {\it uniform} vector bundles, namely those in which the splitting type 
is independent of the chosen line.

Starting from \cite{V}, classification of uniform vector bundles on $\P^n$ has
been developed in a series of papers showing first that if the rank is smaller
than $n$, then the vector bundle splits (cf. \cite{V} for rank two and \cite{S} or
\cite[Thm.~3.2.3]{OSS} for any rank), and characterizing the cases of rank $n$
(see \cite{EHS}) and $n+1$ (see \cite{E} and \cite{B2}). In these {\it low
rank} cases the only uniform vector bundles are constructed upon $T_{\P^n}$ and line bundles by basic operations (tensor products, symmetric or skew-symmetric powers and direct sums). 

Similar results have been obtained for uniform vector bundles on other
projective varieties swept out by lines (for which the notion of uniformity still makes sense) like quadrics (see \cite[Section 4]{KS} and \cite{B1}) and  Grassmannians (see
\cite{G}). Given one of these varieties, one would like to compute first its {\it splitting threshold}, that is the maximum  positive integer $u$ such that any uniform vector bundle of
rank smaller than or equal to $u$ splits as a sum of line bundles, and then classify the low rank cases. For instance, on Grassmannians the only unsplit
uniform vector bundles of rank $u+1$ are twists of the universal
quotient bundle or its dual. These classical results are based on the so called
{\it standard construction} which uses the universal family of lines in $X$ to
produce a uniform vector subbundle of the vector bundle $E$ in such a way that
an induction procedure on the rank works.

More generally, one may consider the problem of classifying uniform vector
bundles on smooth complex projective varieties $X$ dominated by an unsplit
family $\cM$ of rational curves. In the framework of the general theory of rational curves (cf. \cite{K})
the standard construction has a clear geometrical interpretation -- see the proof
of Theorem \ref{thm:criterion} -- allowing us to deal with other uniruled varieties.
In this way, we get a splitting criterion written in terms of the 
cohomology groups of the family $\cM_x$ of curves of $\cM$ by $x \in X$, see Theorem
\ref{thm:criterion}. Moreover, in certain cases, it can be expressed in terms of properties of
the variety of minimal rational tangents to $X$ (see Corollary
\ref{cor:uniformlinear}).

Our criteria may be applied to obtain bounds on the splitting threshold for uniform vector bundles on many Fano
manifolds of Picard number one. In this paper we have tested our techniques on irreducible Hermitian symmetric spaces (see Section
\ref{ssec:hermitian}): we have reproved the classical results for quadrics and
Grassmannians and got new ones for the rest. Furthermore we have proved that, as in the
case of Grassmannians, on the isotropic Grassmannian the universal quotient bundle is  -- up to a twist by a line bundle -- the only uniform unsplit vector bundle of lowest rank. 

Finally we show how to apply our results to the study of Fano bundles.
Recall that a vector bundle $E$ on $X$ is {\it Fano} if $\P(E)$ is a Fano manifold. It is known
that if $E$ is a Fano bundle on $X$ then $X$ is also Fano  \cite[Thm.~1.6]{SW}, and the
problem of classification of Fano bundles on a particular Fano manifold appears
naturally when classifying Fano manifolds. For instance, rank two Fano bundles on projective
spaces or quadrics are completely classified, see \cite{SW}, \cite{SSW} and
\cite{APW}. 

We have dealt here with the problem of low rank Fano bundles on Grassmannians
(see Theorem  \ref{thm:fanograss}), obtaining (see Corollary
\ref{cor:grass1}) that any unsplit Fano bundle of rank two on $\G(1,n)$, $n \geq
5$, is a twist of the universal quotient bundle.
The main idea in the proof is to consider the restriction of the bundle to the maximal
dimensional linear spaces of the Grassmannian and proving that they split as 
direct sums of line bundles. These restrictions are not necessarily Fano, but verify 
a weaker property, namely to be {\it $1$-Fano} (see Definition
\ref{def:rfano}). This property still allows us to use 
 techniques similar to those used in \cite{APW} and conclude that 
$1$-Fano bundles on $\P^n$, $n\geq 4$, are direct sums of line bundles 
(in fact a stronger result is true, see Corollary \ref{cor:splitcrit}). 
In particular Fano bundles on Grassmannians 
are uniform, and then the classification previously obtained finishes the work.

\section{Conventions and definitions.}\label{sec:conventions}
Along this paper we will work with smooth complex projective varieties. 
Given a vector bundle $E$ on a variety $X$, $\P(E)$ will denote the 
Grothendieck projectivization of $E$, i.e.
$$
\P(E)=\Proj\left(\bigoplus_{k\geq 0}S^kE\right).
$$

We will mostly concentrate on Fano manifolds of Picard number one. 
If this is the case, we will denote by $\cO_X(1)$, or by $\cO(1)$ 
if there is no possible confusion, the ample generator of the 
Picard group of $X$. Then, given an algebraic cycle $\alpha$
in $X$ of codimension $m$, we call {\it degree} of $\alpha$
the integer $\deg(\alpha)=\alpha\cdot c_1(\cO_X(1))^m$. A 
rational curve of degree one on a Fano manifold $X$ is 
called a {\it line}.
As usual, the degree of a vector bundle will be the degree 
of its first Chern
class. 

Given a smooth projective variety $X$, we will consider families 
of rational curves on $X$, that is, irreducible components of the 
scheme $\ratcurves^n(X)$ (see \cite[II.2]{K}). We say that a family 
$\cM$ is {\it unsplit} if $\cM$ is a proper $\C$-scheme. Given a 
vector bundle $E$ on $X$ and a family of rational curves $\cM$ 
on $X$, we say that $E$ is {\it
uniform} with respect to $\cM$,
if the restriction of $E$ to the normalization of every curve 
$\ell$ in $\cM$ is isomorphic to
$\cO(a_1)\oplus\dots\oplus \cO(a_r)$, with $a_1\geq\dots\geq\ a_r$ 
fixed. The $r$-tuple $(a_1,\dots,a_r)$ is called the {\it splitting type} 
of $E$ with respect to $\cM$.

Given a smooth variety $X$ and a family of rational curves parametrized 
by $\cM$ such that $\cM_x$ is proper for the
general point $x \in X$, the {\it variety of minimal rational tangents} (VMRT
for short) at $x$ is the closure of the set of points in $\P(\Omega_{X,x})$
corresponding to the tangent lines of the general curves of the family 
$\cM$ passing by $x$. We refer to \cite{Hw} for a complete account on
the VMRT. 

The Grassmannian variety
parametrizing linear projective subspaces of dimension $k$ in $\P^n$, for which we follow the conventions of \cite{A}, will be denoted by $G=\G(k,n)$. 
We will denote by $\cQ$ the rank $k+1$
{\it universal quotient bundle} and by $\mathcal{S}^\vee$ the rank
$n-k-1$ {\it universal subbundle}, related in the universal exact
sequence:
$$0 \to \mathcal{S}^\vee \to \cO^{\oplus (n+1)} \to \cQ \to 0.$$

The projectivization of $\cQ$ provides the universal family of $\P^k$'s
in $\P^n$:
$$\xymatrix{&\P(\cQ)\ar[ld]_{\pi_2}\ar[rd]^{\pi_1}&\\ \G(k,n)&&\P^n.} $$

The variety $G=\G(k,n)$ contains three distinguished families of linear spaces: the
family of lines parametrized by the flag manifold $\F(k-1,k+1,n)$, 
a family of $\P^{k+1}$'s (i.e. subschemes $\P^{k+1}\subset G$ of degree $1$) 
parametrized by a Grassmannian $\G(k+1,n)$ and a family of $\P^{n-k}$'s 
parametrized  by $\G(k-1,n)$. The last two types will be called {\it $\alpha$} and 
{$\beta$-spaces}. Note that in the case $k=n-k-1$, choosing which family is of $\alpha$ or
$\beta$-spaces is equivalent to choosing which universal bundle is the universal
quotient bundle.

The variety $G$ is covered by an unsplit family of lines, parametrized by a 
smooth rational homogeneous space that we
denote by $\F(k-1,k+1,n)$. It is the flag manifold parametrizing chains
$L_{k-1}\subset L_{k+1}\subset\P^n$ of subspaces of dimensions $k-1$ and $k+1$
in $\P^n$.

Given a rank $(n+1)$ matrix $A$, we may consider the subscheme 
$\G_A(k,n)\subset\G(k,n)$ parametrizing isotropic subspaces 
of dimension $k$ in $\P^n$. Later on we will consider the 
{\it isotropic} and {\it symplectic isotropic grassmannians}, 
$QG=\G_Q(m-1,2m-1)$ and 
$LG=\G_L(m-1,2m-1)$, where $Q$ is symmetric and $L$ is 
skew-symmetric, respectively.

We will denote by $\mathbb{Q}^n$ (or just $\mathbb{Q}$ when its dimension
is not relevant) the $n$-dimensional smooth quadric.

Finally, given any real number $a$, we will denote by $\lfloor a\rfloor$ 
(resp. by $\lceil a \rceil$) its round-down (resp. round-up). 

\section{Splitting criteria for vector bundles on Fano
manifolds}\label{sec:criterion}

In this section we present a splitting criterion for vector bundles of small
rank on varieties covered by lines. Its proof goes through the construction of a
quotient $F$ of the vector bundle $E$ which is trivial on lines. The classical
argumentation for a statement of this type relies on a descent lemma due to
Forster, Hirschowitz and Schneider (cf. \cite{FHS}), that reduces the existence
of $F$ to the vanishing of a certain group of cohomology. This was the point
of view of, for instance, Guyot's classification of uniform vector bundles on
grassmannians, \cite{G}. The use of the descent lemma may be substituted in some
cases by a more precise argument in which the concrete cycles that control the
descent, and not the whole group of cohomology is considered. This idea appears
already in the classification of uniform vector bundles on quadrics due to Kachi
and Sato (cf. \cite[Thm. 4.1]{KS}). In this paper we use
a geometric reformulation of this argument involving rational curves. We will also 
make use of the characterization of uniform vector bundles of constant splitting
type on Fano manifolds due to Andreatta and Wi\'sniewski (cf. \cite[Prop.~1.2]{AW}).

Let $X$ be a Fano manifold of Picard number one, denote by $\cO(1)$ the ample
generator of $\Pic(X)$
and assume that $X$ admits an unsplit covering family of rational curves $\cM$.
Denoting by $\cU$ the 
universal family, which has a natural $\P^1$-bundle structure over $\cM$, 
we get the diagram:
$$
\xymatrix{\cU\ar^{p_1}[r]\ar^{p_2}[d]&X\\\cM&}
$$
Given any subset $Y\subset X$ we denote $\cM_Y:=p_2p_1^{-1}(Y)$, and by
$\Loc(\cM)_Y$ the subscheme
$p_1p_2^{-1}(\cM_Y)$ of points lying in curves of the family $\cM$ meeting $Y$.
Note that for every point $x$ the variety $\cM_x$ is proper by hypothesis.

\begin{thm}\label{thm:criterion}
 Let $X$ be a Fano manifold of Picard number one admitting an unsplit covering
family of rational curves 
$\cM$ and let $r$ be a positive integer smaller than
or equal to $\dim\cM_x$. Assume that
$\dim H^{2s}(\cM_x,\C)=1$
for all $s\leq \lfloor r/2\rfloor$ and
every
$x\in X$. Then the only uniform vector bundles of rank $r$ on $X$ are direct
sums of line bundles.
\end{thm}

\begin{rem}\label{rem:criterion1}
 {\rm The criterion can be easily checked for many examples of Fano manifolds,
e.g. projective spaces, quadrics and others.
Note that if $X$ admits an embedding for which the curves parametrized by $\cM$
are lines, then $\cM_x$ coincides with the VMRT of $X$ at $x$.
For projective spaces, Theorem \ref{thm:criterion} tells us that the only uniform
vector bundles on $\P^n$ of rank smaller than $n$ are direct sums of line bundles. The result is sharp, since
$T_{\P^n}$ is unsplit, and it was first obtained in \cite{V} and
\cite{S}.}
\end{rem}

In the case of quadrics, Theorem \ref{thm:criterion} provides the same statement as in \cite[Thm. 4.1]{KS}.

\begin{cor}\label{cor:critquad}
Let $\Q$ be a smooth quadric of dimension $n$, with $n\geq 5$ odd
(resp. even). The only uniform vector bundles on $\Q$
of rank smaller than or equal to $n-2$ (resp. $n-3$)
are direct sums of line bundles.
\end{cor}

\begin{proof}
Note that in this case $\cM_x=\Q^{n-2}$ (cf. \cite{Hw}) and denote by 
$s_0=\min\{s,\,\dim H^{2s}(\cM_x,\C)\neq 1\}-1$,
which is then equal
to $n-2$ if $n$ is odd and $(n-2)/2-1$ if $n$ is even. If $n$ is odd
we may apply the criterion to every vector bundle of rank smaller
than or equal to $n-2=\dim\cM_x=s_0$. If $n$ is even we need to
assume also that $\lfloor r/2\rfloor\leq s_0=(n-4)/2$, i.e. $r\leq n-3$.
\end{proof}

\begin{proof}[Proof of Theorem \ref{thm:criterion}]
Let $E$ be a vector bundle of rank $r$. If the splitting type of $E$ is constant, 
then it is a direct sum of line bundles by \cite[Prop.~1.2]{AW}. Hence we may 
assume, after dualizing if necessary, that $E$ has splitting type
$$
(a_1,\dots,a_k,\dots, a_r)\mbox{, with }a_1=\dots=a_k<a_{k+1}\leq\dots\leq
a_r,
$$
where $k\leq \lfloor r/2\rfloor$.

We will prove that $E$ splits by induction on $r$. Since the only Fano variety 
of dimension $1$ is $\P^1$, for which the result is known, we
may assume that $\dim X>1$. Then
Kodaira Vanishing Theorem allows us to reduce the proof to showing that $E$
admits a surjective morphism onto a direct sum of line bundles.

Let us first construct the family of minimal sections of $E$ over curves of
$\cM$. With
this we mean the family of rational curves on $\P(E)$ determined by surjective
maps $E|_\ell\to\cO_\ell(a_1)$, where $\ell\in\cM$.

Denote by $\cO_\cU(1)$ the tautological line bundle of the $\P^1$-bundle $p_2:\cU\to\cM$.
With the same notation as above, denote by $F$ the vector bundle
$$F:=\left({p_2}_*\left(p_1^*E\otimes\cO_\cU(-a_1)\right)^\vee\right)^\vee,$$
$\cM_E:=\P(F)$ and
$\cU_E:=\P(p_2^*F)$. Note that every element of $F^\vee$ corresponds to a pair
$(\ell,s)$, where $\ell$ is a curve of $\cM$ and $s$ is a morphism
$E|_\ell\to\cO_\ell(a_1)$. The natural surjective morphism $p_1^*E\to p_2^*F\otimes\cO_\cU(a_1)$
provides an injection $\P(p_2^*F)\hookrightarrow\P(p_1^*E)$ and we get a
commutative diagram:
$$
\xymatrix{\cU_E\ar[dd]^{\widetilde{p_2}}\ar[dr]\ar[r]\ar@/^1pc/[rr]^{
\widetilde{p_1}}&\P(p_1^*E)\ar[d]\ar[r]&\P(E)\ar[d]^{\pi}\\
&\cU\ar[r]^{p_1}\ar[d]^{p_2}&X\\
\cM_E\ar[r]&\cM&}
$$

For every point $x\in X$ the map $\widetilde {p_1}$ provides a morphism
$$\widetilde{p_1}^x:\cM_x\to\G(k-1,\P(E_x)),$$ which is constant by Lemma
\ref{lem:constant} below. Note that $k\leq\lfloor r/2\rfloor$, so by hypothesis 
$H^{2s}(\cM_x,\C)$ has dimension one for every $s\leq k$.

In this way the image of $\widetilde{p_1}$ is a
projective subbundle $\P(E_0)\subset\P(E)$, and the vector bundle $E_0$ is a
quotient of $E$ that has constant splitting type on every curve $\ell\in\cM$. 
Then we may apply
\cite[Prop.~1.2]{AW} in order to conclude that $E_0$ is isomorphic to 
$\cO_X(b)^{\oplus k}$ for some $b$. This concludes the proof.
\end{proof}

\begin{lem}\label{lem:constant}
 Let $k<r$ be positive integers and $M$ be a projective variety of dimension
bigger than or equal to $r$ verifying that
$\dim H^{2s}(M,\C)=1$
for every $s=1,\dots,k$. 
Then the only morphisms from $M$ to Grassmannians $\G(k-1,r-1)$
are constant.
\end{lem}

\begin{proof}
Let $\varphi:M\to\G(k-1,r-1)$ be a morphism and consider the restrictions
$\varphi^*\cS$, $\varphi^*\cQ$ of the dual of the universal bundle and of the universal
quotient
bundle on $\G(k-1,r-1)$, respectively. Denote by $d_1,\dots,d_{r-k}$ and by
$c_1,\dots,c_k$ their  Chern classes. Note that our hypothesis on
the cohomology 
of $M$ allows us to identify the $c_i$'s with integers. 

 From the exact sequence:
$$
0 \to \varphi^*\cS^\vee \to \cO_M^{\oplus n+1} \to \varphi^*\cQ \to 0,
$$
we get the equality of polynomials:
$$1=\sum_{i=0}^kc_it^i\cdot\sum_{j=0}^{r-k}(-1)^jd_jt^j.$$

Note that in principle this equality is satisfied only modulo $t^{\dim M}$; that
is why we need to assume that $\dim M\geq r$.
We may write this equation in the following way:
$$
\underbrace{\left(\begin{array}{ccccc}
c_1&1&\cdots&\cdots&0\\
\vdots&c_1&\ddots& &\vdots\\
c_{k}&\vdots&\ddots&\ddots&\vdots\\
0&c_{k}& &\ddots&1\\
\vdots&\vdots&\ddots& &c_1\\
\vdots&\vdots& &\ddots& \vdots\\
0&0&\cdots&\cdots &c_{k}\\
\end{array}\right)}_A
\left(\begin{array}{c}
1\\-d_1\\ \vdots\\ \vdots\\(-1)^{r-k}d_{r-k}
\end{array}\right)=0
$$
and we see that if any $c_i$ is different from zero then the matrix $A$ has
maximal rank, a contradiction. Therefore $c_i=0$ for all $i>0$. In particular
$0=c_1=\deg\varphi^*\cO_{\G(k-1,r-1)}(1)$, hence $\varphi$
is constant.
\end{proof}

Note that the conclusion of Lemma \ref{lem:constant} also holds if we only
assume that $M$ is chain-connected by projective varieties of 
dimension bigger than or equal to $r$ and whose 
$2s$-th cohomology group has dimension
one for all $s=0,\dots,\lfloor r/2\rfloor$. In particular we obtain the
following:

\begin{cor}\label{cor:criterion2}
Let $X$ be a Fano manifold of Picard number one covered by an unsplit family of
rational curves $\cM$ and let $r$ be a positive integer. Assume further that for
every $x\in X$,
$\cM_x$ is chain-connected by projective varieties of dimension bigger than or
equal to $r$ whose 
$2s$-th cohomology group with complex coefficients has dimension
one for all
$s=0,\dots,\lfloor r/2\rfloor$. Then $X$ does not admit uniform vector
bundles of rank smaller than $r$, apart
of direct sums of line bundles.
\end{cor}

It is classically known that the VMRT at every point of the Grassmannian 
$\G(k,n)$ is the Segre embedding of $\P^{k}\times\P^{n-k-1}$ in $\P(\Omega^1_{G})_x=\P(\cQ^\vee\otimes\cS)_x$. We refer the interested 
reader to \cite{LM} for a description of the VMRT of rational homogeneous spaces 
covered by lines. See also 
\cite[Section~1.4]{Hw}, where the VMRT's of many examples of Fano manifolds are shown.
Since $\P^{k}\times\P^{n-k-1}$ is chain-connected by linear subspaces of
dimension $\min\{k,n-k-1\}$, we obtain the following:

\begin{cor}\label{cor:critgrass}
 The only uniform vector bundles of rank $r<\min\{k+1,n-k\}$ on a Grassmannian
$\G(k,n)$ are direct sums of line bundles.
\end{cor}

More generally, for Fano manifolds covered by linear spaces, we may state a
splitting criterion independent on the rank:

\begin{cor}\label{cor:uniformlinear}
Let $X\subset\P^N$ be a Fano manifold of Picard number one
covered by a family $\cL$ of linear subspaces
of dimension $d\geq 2$, and assume that at every point $x\in X$ the VMRT of $X$
at $x$ associated to the family of lines
is chain-connected by the corresponding linear spaces of dimension $d-1$.
Let $E$ be a vector bundle on $X$ verifying that $E|_L$ is a direct sum of line
bundles for every $L$ of $\cL$. Then $E$ is a direct sum of line bundles.
\end{cor}

\begin{proof}
We argue as in Theorem \ref{thm:criterion}, and use the same notation. Given a
linear space $L$ of the family $\cL$ and a line $\ell\subset L$, take any point
$x\in \ell$ and denote by $L'$ the linear subspace of the VMRT of $X$ at $x$
corresponding to tangent directions to $L$ at $x$. By hypothesis any surjective
morphism $E|_\ell\to\cO_\ell$ lifts to a unique $E_L\to\cO_L$. In particular the
morphism $\widetilde{p_1}^x:\cM_x= \text{VMRT}_x\to \G(k-1,r-1)$ constructed as in the
theorem is constant on $L'$, and so it is constant on the VMRT.
\end{proof}

\subsection{Uniform vector bundles on Hermitian symmetric
spaces}\label{ssec:hermitian}

We have already applied our criteria to uniform vector bundles on quadrics and
Grassmannians. In this section we show, as an example, how our methods work for
other Hermitian symmetric spaces. We refer the reader to \cite{LM} or 
\cite{Hw} for the description of their VMRT's. 

Let us introduce the following notation. Given a Fano manifold $X$ of
Picard number one and an unsplit covering family of rational curves $\cM$ we
will denote by $u(X,\cM)$, or simply by $u(X)$ if the family $\cM$ is clear, the
maximum positive integer verifying that every uniform vector bundle on $X$ of
rank $\leq u(X)$ is a direct sum of line bundles, and we call it the {\it
splitting threshold for uniform vector bundles} on $X$. If $X$ is a homogeneous
Fano manifold of Picard number one, then $T_X$ is an unsplit homogeneous bundle, 
in particular $u(X)<\dim X$. 

In the following table we show the possible values of $u(X)$ for Hermitian symmetric 
spaces. The last column shows unsplit and uniform vector bundles $F$ of rank $u(X)+1$:

\begin{table}[h]
	\centering
		\begin{tabular}{|c|c|c|c|c|}
		\hline
			$G$&$X$& $VMRT$&$u(X)$&$F$\\\hline
			$\Sl(n+1)$& $\P^n$ &$\P^{n-1}$&$n-1$&$T_{\P^n}$\\
			$\Sl(n+1)$&$\G(k,n)$ ($k\leq n/2$)& 
$\P^{k}\times\P^{n-k-1}$&$k+1$&$\cQ$\\
%$\SO(2k+2),\,k\geq 3$&$\Q^{2k}$&$\Q^{2k-2}$&$\geq 2k-3$\\
%$\SO(2k+1),\,k\geq 3$&$\Q^{2k-1}$&$\Q^{2k-3}$&$\geq 2k-3$\\
			$\SO(2m)$&$\G_Q(m-1,2m-1)$
&$\G(1,m-1)$&$m-1$&$\cQ$\\
			$\Sp(2m)$&$\G_L(m-1,2m-1)$& 
$v_2(\P^{m-1})$&$m-1$&$\cQ$\\
			$\mbox{E}_6$&$\O\P^2\otimes_\C\R$&  $\G_Q(4,9)$&$\geq
5$&\\
			$\mbox{E}_7$&$X$& $\O\P^2\otimes_\C\R$&$\geq
7$&\\
		\hline
		\end{tabular} %\vspace{0.2cm}
	\caption{Splitting threshold for uniform vector bundles on Hermitian
symmetric spaces}
	\label{tab:Hermitian}
\end{table}

The case of $\P^n$ is classical (cf. \cite{S}), but we recover it as follows: 
Theorem \ref{thm:criterion} tells us that $u(\P^n)\geq n-1$ 
and on the other hand the tangent bundle is unsplit and uniform, hence $u(X)$ 
equals $n-1$. In a similar way, Theorem \ref{thm:criterion} provides 
the estimation presented in the table for the symplectic Grassmannian, 
which is sharp by the unsplitting of the universal quotient bundle $\cQ$. 

For the Grassmannian $\G(k,n)$ (for simplicity we have chosen 
here $k\leq n/2$) the value of $u(X)$ follows from Corollary 
\ref{cor:critgrass} and the unsplitting of $\cQ$, again. 

The estimation obtained for the Hermitian symmetric 
spaces corresponding to $\mbox{E}_6$, and $\mbox{E}_7$ has been obtained, 
once we know the cohomology of their VMRT's, by 
applying directly Theorem \ref{thm:criterion}. In the first case, the 
corresponding VMRT is $\G_Q(4,9)$, which verifies (see, for instance, \cite{Di}) that 
$$
H^2(\G_Q(4,9),\C)\cong H^4(\G_Q(4,9),\C)\cong\C\not\cong H^6(\G_Q(4,9),\C).
$$
For the space corresponding to $E_7$ we need to know the cohomology of the Cayley 
plane $\mathcal{C}:=\O\P^2\otimes_\C\R$, that verifies  (cf. \cite{IM}): 
$$H^2(\mathcal{C},\C)\cong H^4(\mathcal{C},\C)\cong 
H^6(\mathcal{C},\C)\cong\C\not\cong H^8(\mathcal{C},\C).$$
Let us remark that the same result is obtained by applying Barth-Larsen 
Theorem (cf. \cite[Thm.~3.2.1]{L}) to the usual embedding of these two 
varieties in a projective space.

Finally, we consider the isotropic grassmannian $\G_Q(m-1,2m-1)$ 
(parametrizing any of the two
families of $\P^{m-1}$'s contained in a quadric of dimension $2m-2$). Note that the
universal quotient bundle on $\G_Q(m-1,2m-1)$ is unsplit and has rank $m$, hence
$u(\G_Q(m-1,2m-1))<m$. Thus the value of $u$
follows from the following proposition:

\begin{prop}\label{prop:quadgrass}
Every uniform vector bundle on
$\G_Q(m-1,2m-1)$ of rank smaller than or equal to $m-1$ is a direct sum of line
bundles.
\end{prop}

The proof of this result follows from Lemma \ref{lem:quadricgrass1} below, in 
the same way as Theorem \ref{thm:criterion} follows from
Lemma \ref{lem:constant}. 

\begin{lem}\label{lem:quadricgrass1}
There are no nonconstant maps from $\G(1,n)$ to $\G(k-1,n-1)$, for every
$k\in\{1,\dots,n-1\}$. 
\end{lem}

\begin{proof}
Assume we have a nonconstant morphism $\phi:\G(1,n)\to\G(k-1,n-1)$, and denote
by $\psi$ its restriction to a linear space $\P^{n-1}\subset\G(1,n)$. Note that
since $\phi$ is nonconstant, then $\psi$ is nonconstant, too. Denote by
$c_1,\dots,c_k$ and $d_1,\dots,d_{n-k}$ the Chern classes of $\psi^*\cQ$ and
$\psi^*\cS$, respectively, and by $C_1,\dots,C_k$ and $D_1,\dots,D_{n-k}$ the
Chern classes of $\phi^*\cQ$ and $\phi^*\cS$, respectively.

Arguing as in Lemma \ref{lem:constant} we obtain an equation
$$
\left(\begin{array}{ccccc}
c_1&1&\cdots&\cdots&0\\
\vdots&c_1&\ddots& &\vdots\\
c_{k-1}&\vdots&\ddots&\ddots&\vdots\\
c_k&c_{k-1}& &\ddots&1\\
\vdots&c_k&\ddots& &c_1\\
\vdots&\vdots&\ddots&\ddots& \vdots\\
0&0&\cdots&c_{k} &c_{k-1}\\
\end{array}\right)
\left(\begin{array}{c}
1\\-d_1\\ \vdots\\ \vdots\\(-1)^{n-k}d_{n-k}
\end{array}\right)=0
$$
Note that we may assume that $c_k$ and $d_{n-k}$ are different from zero,
otherwise the above equation would imply $c_1=0$, contradicting that $\psi$ is
nonconstant.

On the other side, a similar argumentation with $\phi$ tells us that $C_k\cdot
D_{n-k}=0$. In order to show that this contradicts that $c_k$ and $d_{n-k}$ are
non-zero, we may use some Schubert calculus. Let us fix some notation first.

We will denote by $Z^k_i$ the Schubert cycle in $\G(1,n)$ of codimension $k$
determined by the partition $(k-i,i)$, i.e. the cohomology class of the set
parametrizing lines of $\P^{m}$ contained in a fixed $\P^{m-i}$ and meeting a
fixed $\P^{m-1-k+i}\subset\P^{m-i}$.

With this notation we may write
$$
C_k=\sum_{i=0}^{\lfloor k/2\rfloor}a_iZ^k_i,\quad D_{n-k}=\sum_{j=0}^{\lfloor (n-k)/2\rfloor}b_jZ^{n-k}_j,
$$
where the $a_i$'s and the $b_j$'s are non negative integers by the nefness of
$\cQ$ and $\cS$, cf. \cite{BG} and \cite[8.2]{L}.
Pieri's formula tells us that
$$c_k=C_k\cdot\P^{n-1}=a_0Z_k^{n-1-k} \mbox{ and }
d_{n-k}=D_{n-k}\cdot\P^{n-1}=b_0Z_{n-k}^{2n-1-k},$$
hence $a_0b_0\neq 0$. But on the other side:
$0=C_k\cdot D_{n-k}=a_0b_0Z_0^kZ^{n-k}_0+\dots$, where all the elements of this
sum are linear combination of Schubert cycles with non negative coefficients, by
the Littlewood-Richardson formula. In particular it follows that $a_0b_0=0$, a
contradiction.
\end{proof}

%%%%%%%%%%%%%%%%%%%%%%%%%%%%%%%%%%%%%%%%%%%%%%%%%%%%%%%%%%%%

\section{Uniform vector bundles on Grassmannians}\label{sec:uniform}
In this section we use some geometric arguments to go one step further from
Theorem \ref{thm:criterion} in the classification of uniform vector bundles of
low rank on Grassmannians. This result was previously obtained by Guyot,
\cite{G}.

Applying Corollary \ref{cor:critgrass} to Grassmannians we already know that
the
only uniform vector bundles on $\G(k,n)$ of rank smaller than $\min\{k+1,n-k\}$
are direct sums of line bundles. In this section we prove that, up to twist by a line bundle, the only unsplit
uniform vector bundle of rank equal to $\min\{k+1,n-k\}$ is the universal bundle
of smaller rank. Summing up we obtain the following:

\begin{thm}\label{thm:uniformgrass}
Let $G=\G(k,n)$ be the Grassmannian variety
parametrizing linear projective subspaces of dimension $k$ in $\P^n$, $k\leq
n-k-1$, and let $E$ be a uniform vector bundle on $G$ of rank $r\leq k+1$. Then
$E$ is a direct sum of line bundles unless it is a twist of either the universal
quotient bundle $\cQ$ or its dual $\cQ^\vee$.
\end{thm}

By assumption $E$ is
uniform of rank $r\leq k+1$ on $G$, hence we may restrict to $\alpha$ and
$\beta$-spaces and use the classification of uniform vector bundles on
projective spaces \cite[Thm.~3.2.3]{OSS}. Note that uniform vector bundles on
projective spaces of rank smaller than or equal to the dimension are completely
determined by their Chern classes. This implies that the restrictions of $E$ to
two $\alpha$ (respectively $\beta$)-spaces are isomorphic, and we may
distinguish three cases:
\begin{enumerate}
\item The restriction of $E$ to every $\alpha$ and $\beta$-space is a direct sum
of line bundles. This case has been already considered in Corollary
\ref{cor:uniformlinear}.
\item The restriction $E|_L$ is a twist of $T_L$ or $\Omega_L$ for every
$\alpha$-space $L$ and $E|_M$ is a direct sum of line bundles for every
$\beta$-space $M$.
\item $k=r-k-1$ and the restrictions of $E$ to every $\alpha$ or $\beta$-space
$L$ are isomorphic to a twist of $T_L$ or $\Omega_L$.
\end{enumerate}

Our first task will be discarding the third case.

\begin{claim}\label{claim:TT}
 With the same notation as above, the restriction of $E$ to either an $\alpha$
or a $\beta$-space is a direct sum of line bundles.
\end{claim}

\begin{proof}
On the contrary, we will assume that $k=r-k-1$ and the restrictions of $E$ to
every $\alpha$ and $\beta$-space
are unsplit. By the classification of uniform vector bundles of rank $k$ on 
$\P^k$ (cf. \cite{EHS}) we may assume, after dualizing and/or
twisting $E$ with an appropriate line bundle $\cO_G(a)$, that
$E|_L$ is isomorphic to $\Omega_L(2)$, for every $\alpha$ or $\beta$-space $L$.
The main idea of the proof is the following: For every point $g\in G$ and every 
$\alpha$ or $\beta$-space containing it we obtain, in the same way as in the 
proof of Theorem \ref{thm:criterion} a map from $\P^{k-1}$ to $\P(E_g)$ which 
is an isomorphism by hypothesis. We will show that the existence of these 
isomorphisms for every $\alpha$ and $\beta$-space passing by $g$ leads to 
a contradiction. 

 Let us consider two $\alpha$-spaces $L_1$ and $L_2$ meeting at one point $g$,
and denote by $P_1$ and $P_2$ the projective spaces parametrizing lines through
$g$ in $L_1$ and $L_2$, respectively. By Lemma \ref{lem:familiesgrass} below, there is
an isomorphism $\phi:P_1\to P_2$, verifying that every pair of lines $r$, $\phi(r)$ lie on a $\beta$-space.

On the other hand, consider the restriction of $E$ to any line $\ell_i\subset
L_i$ passing by $g$. Since $E|_{\ell_i}$ is isomorphic to
$\Omega_{L_i}(2)|_{\ell_i}\cong\cO\oplus\cO(1)^{\oplus k}$, the
unique surjective map from it onto $\cO$ determines a point
$\psi_i(\ell_i)\in\P(E_g)$. In fact, since $E_g\cong\Omega_{L_i,g}$
then the map $\psi_i:P_i\to\P(E_g)$ is a linear isomorphism. The map
$\psi:=\psi_2^{-1}\circ\psi_1$ gives a second isomorphism from $P_1$ to $P_2$,
that we want to study together with $\phi$.

A fixed point for $\psi^{-1}\circ\phi$ provides a line $\ell\in P_1$ verifying
$\phi(\ell)=\psi(\ell)$, i.e. $\psi_1(\ell)=\psi_2(\phi(\ell))$. But $\ell$
and $\phi(\ell)$ are contained in a $\beta$-space $M$. In the same way as for
$L_i's$, the minimal sections over lines in $M$ through $g$ provides a linear
isomorphism between the set of lines through $g$ in $M$ and
$\P(E_g)$, contradicting that $\psi_1(\ell)=\psi_2(\phi(\ell))$.
\end{proof}

We have made use of the following lemma, which is a 
straightfoward consequence of the fact that the VMRT of 
$G$ at any point $g$ is isomorphic to the Segre variety 
$\P^{k}\times\P^{n-k-1}$, whose maximal linear spaces
correspond to the $\alpha$ and $\beta$-spaces in $G$ through $g$.

\begin{lem}\label{lem:familiesgrass}
 With the same notation as above, let $L_1$ and $L_2$ be two $\alpha$ 
 (resp. $\beta$)-spaces meeting at one point $g$. Given any line 
 $\ell_1\subset L_1$ passing by $g$, there exists a
$\beta$ (resp. $\alpha$)-space $M$ meeting $L_1$ in $\ell_1$ and $L_2$ in 
a line $\ell_2$. This construction provides a linear isomorphism $\phi$ between 
the set of lines in $L_1$ by $g$ and the set of lines in $L_2$ by $g$.
 \end{lem}

Finally we will show that the only vector bundles in case (2) are twists
of
the universal quotient bundle or its dual.

\begin{prop}\label{prop:decomp}
 With the same notation as above, assume that the restriction $E|_L$ is a twist
of $T_L$ or $\Omega_L$ for every
$\alpha$-space $L$ and $E|_M$ is a direct sum of line bundles for every
$\beta$-space $M$. Then $E$ is a twist of either the universal
quotient bundle $\cQ$ or its dual $\cQ^\vee$.
\end{prop}

\begin{proof}

Twisting with an appropriate power of $\cO_G(1)$ and dualizing if necessary we
may assume that the restriction of $E$ to an $\alpha$-space $L$ is isomorphic to
$\Omega_L(2)$ and to a $\beta$-space $M$ isomorphic to
$\cO_M\oplus\cO_M(1)^{\oplus k}$.

We will consider the family of $\beta$-spaces:
\begin{equation}
 \xymatrix{&\F(k-1,k,n)\ar[ld]_{p_2}\ar[rd]^{p_1}&\\
 \G(k-1,n)&&\G(k,n)}\label{eq:beta}
\end{equation}

Arguing as in
Theorem \ref{thm:criterion}, the subsheaf
$F^\vee:=p_2^*{p_2}_*p_1^*E^\vee\subset p_1^*E^\vee$ is a line bundle, whose
restriction to any fiber of $p_2$ is trivial. It provides a  surjective
morphism $p_1^*E\to F$, and so a morphism $\psi$ fitting in the
following commutative diagram:

$$\xymatrix{&\F(k-1,k,n)\ar[ld]_{p_2}\ar[rd]^{p_1}\ar[r]^(.60)\psi&\P(E)\ar[d]^{\pi}
\\
    \G(k-1,n)& &\G(k,n)}$$
The morphism $\psi$ may be also described in the following way: its restriction
$\psi|_{p_2^{-1}(x)}$ is determined by the
only surjective morphism $E|_{p_1(p_2^{-1}(x))}\to\cO$, that provides a section
$\varphi$ of $\pi$ over $p_1(p_2^{-1}(x))$, and
$$
\psi|_{p_2^{-1}(x)}=\varphi\circ p_1.
$$
Note that $\F(k-1,k,n)$ is isomorphic to $\P(\cQ^\vee(1))$ and its dimension
equals $\dim\P(E)$, hence it suffices to show that $\psi$ is injective.

Given two $\beta$-spaces
$M_1$ and $M_2$ meeting at a point $g$, the images of the corresponding liftings
$\varphi_1$ and $\varphi_2$
meet $\P(E_g)$ in two points, $y_1$ and $y_2$, respectively. We may take an
$\alpha$-space $L$ meeting them in two lines $\ell_1=L\cap M_1$ and
$\ell_2=L\cap
M_2$, as described in Lemma
\ref{lem:familiesgrass}. But $E|_L\cong\Omega_L(2)$, then
$$
y_1=\varphi_1(\ell_1)\cap\P(E_g)\neq\varphi_2(\ell_2)\cap\P(E_g)=y_2,
$$
because $\varphi_i(\ell_i)\cap\P(E_g)$ is the point in $\P(E_g)$ corresponding
to the tangent direction to $\ell_i$ at $g$ in $L$. This concludes the proof.
\end{proof}

A similar argument, together with the classification of uniform vector bundles
of rank $n+1$ on $\P^n$
(cf. \cite{E} and \cite{B2}), can be applied to isotropic Grassmannians:
\begin{prop}\label{prop:quadricgrass2}
Every uniform vector bundle on the isotropic Grassmannian $\G_Q(m-1,2m-1)$ of rank
$m$ is decomposable unless it is a twist of the universal quotient bundle $\cQ$
or its dual.
\end{prop}

\begin{proof}
We may assume $m\geq 4$, since on $\G_Q(1,3)\cong\P^1$ and on $\G_Q(2,5)\cong\P^3$
the result is well-known. Recall that $QG:=\G_Q(m-1,2m-1)$ is covered by
$\P^{m-1}$'s, and the corresponding family is parametrized by $QG$:
$$\xymatrix{&\G_Q(m-2,2m-1)\ar[ld]_{p_2}\ar[rd]^{p_1}&
\\QG& &QG}
$$
Moreover $\G_Q(m-2,2m-1)$ equals the projectivization of the dual of the universal
quotient bundle $\cQ$. Hence, it suffices to show that $\P(E)$ or $\P(E^{\vee})$
are isomorphic to $\G_Q(m-2,2m-1)$ as $QG$-schemes.

Let $E$ be a uniform vector bundle on $QG:=\G_Q(m-1,2m-1)$ and assume that $E$ is
not a direct sum of line bundles.

Given any $\P^{m-1}\subset QG$ the restriction $E|_{\P^{m-1}}$ cannot be
decomposable. Otherwise, since rank $m$ uniform vector bundles on $\P^{m-1}$ are
determined by their Chern classes (see \cite{E} and \cite{B2}), it would be
decomposable for every $\P^{m-1}$, and $E$ would be decomposable by Corollary
\ref{cor:uniformlinear}. It follows that, up to dualizing and twisting, $E$ is
isomorphic to $\Omega_{\P^{m-1}}(2)\oplus\cO_{\P^{m-1}}(a)$, cf. \cite{E}
and \cite{B2}. We claim that $a=0$. In fact if it is not one easily sees that,
after dualizing and/or twisting conveniently, the splitting type of $E$ would
take the form $(0=a_0<a_1\leq...)$. Arguing as in Theorem \ref{thm:criterion} we
would get a morphism $\G(1,m-1)\to\P^{m-1}$ for every point of $QG$. Since these
morphisms are necessarily constant we would get a surjective map from $E$ onto a
line bundle, whose kernel, by Proposition \ref{prop:quadgrass} is a direct
sum of line bundles. It would follow that $E$ is decomposable, a contradiction.

As a consequence, given any $\P^{m-1}\subset QG$, there is a unique surjective
map $E|_{\P^{m-1}}\to\cO_{\P^{m-1}}$, providing a unique minimal section
$\P^{m-1}\hookrightarrow\P(E|_{\P^{m-1}})$. Arguing as in Proposition
\ref{prop:decomp} we may lift the family of $\P^{m-1}$'s in $QG$ to a family in
$\P(E)$:
$$\xymatrix{&\P(\cQ^\vee)\ar[ld]_{p_2}\ar[rd]^{p_1}\ar[r]^{\psi}&\P(E)\ar[d]^{
\pi}
\\QG& &QG}
$$

The proof is finished by showing that $\psi$ is an isomorphism. Since $\psi$ is
a morphism of $QG$-schemes, it suffices to show that the restriction
$\psi|_{p_1^{-1}(x)}:p_1^{-1}(x)\to\pi^{-1}(x)$ is an isomorphism for every $x$.
Fix a point $x\in QG$ and let $M$ be a $\P^{m-1}$ containing it. Denote by
$\varphi$ the restriction:
$$
\varphi:=\psi|_{p_1^{-1}(M)}:p_1^{-1}(M)\cong\P\left(\Omega_{M}(2)\oplus\cO_{M}
\right)\to \pi^{-1}(M)\cong\P\left(\Omega_{M}(2)\oplus\cO_{M}\right),
$$
which is a morphism of $M$-schemes. By Lemma \ref{lem:quadricgrass} below
$\varphi$ is either a linear isomorphism at every fiber or it factors via
$p_1|_{p_1^{-1}(M)}$. But if the former happens for a $\P^{m-1}$, then $\psi$
contracts vertical lines, and consequently it factors through $p_1$, providing a
section of $\pi$. In this case we would get a surjection from $E$ to a line
bundle, whose kernel would be decomposable by Proposition
\ref{prop:quadgrass}, and $E$ would be decomposable, too.
\end{proof}

\begin{lem}\label{lem:quadricgrass}
Let $\varphi:\P\left(\Omega_{\P^n}(2)\oplus\cO_{\P^n}\right)\to
\P\left(\Omega_{\P^n}(2)\oplus\cO_{\P^n}\right)$ be a morphism of
$\P^n$-schemes, $n\geq 3$. Then either $\varphi$ is an isomorphism or it factors
through the natural projection
$\pi:\P\left(\Omega_{\P^n}(2)\oplus\cO_{\P^n}\right)\to\P^n$.
\end{lem}

\begin{proof}
Let us begin by recalling some features on
$X:=\P\left(\Omega_{\P^n}(2)\oplus\cO_{\P^n}\right)$. The unique surjection
$\Omega_{\P^n}(2)\oplus\cO_{\P^n}\to\cO_{\P^n}$ determines a distinguished
section of $\pi:X\to\P^n$ whose image we denote by $S$.
Given a line $\ell\subset\P^n$, a surjective morphism
$(\Omega_{\P^n}(2)\oplus\cO_{\P^n})_\ell\cong\cO_\ell(1)^{n-1}
\oplus\cO_\ell^2\to\cO$ provides a section of $X:=\P(E)$ over $\ell$, whose
numerical class determines a extremal ray of $\NE(X)$ and a contraction
$\tau:X\to Y$, where $Y$ is a cone with vertex one point and base the
Grassmannian $\G(1,n)$. Every fiber of $\tau$ is a minimal section of $\P(E)$
over a line, except for the fiber over the vertex of $Y$, that is $S$. Given a
point $x\in\P^n$ and a line $\ell\ni x$, the minimal sections of $\pi$ over
$\ell$ cut the fiber $\pi^{-1}(x)$ at the points of a line
$r(\ell,x)\subset\pi^{-1}(x)$. Moreover the line $r(\ell,x)\subset\pi^{-1}(x)$
contains the point $O_x:=\pi^{-1}(x)\cap S$.

Given a morphism $\varphi:X\to X$ of $\P^n$-schemes, its restriction
$\varphi_x:=\varphi|_{\pi^{-1}(x)}$ is either constant or surjective onto
$\pi^{-1}(x)$. If the former happens for a point $x\in\P^n$, then $\varphi$
factors via $\pi$. Therefore we may assume that $\varphi|_{\pi^{-1}(x)}$ is
surjective for every $x$, and it suffices to show that in this case $\varphi_x$
is linear for some $x$.

In this case the composition $\tau\circ\varphi$ factors through $\tau$
necessarily (just take the Stein factorization of the morphism $\tau \circ
\varphi$)
$$
\xymatrix{X\ar[r]^\varphi\ar[d]_\tau&X\ar[d]^{\tau}\\Y\ar[r]^{\overline \varphi}&Y}
$$
and consequently the image of a minimal section of $\P(E)$ over a line $\ell$ is
a minimal section over $\ell$, too. In particular, if the degree of $\overline \varphi$ is
bigger than one, then the preimage of a general minimal section $r$ would
contain at least two minimal sections $r_1$ and $r_2$ verifying
$\pi(r)=\pi(r_1)=\pi(r_2)$. In particular $r_1\cap \pi^{-1}(x)$ and
$r_2\cap\pi^{-1}(x)$ lie in $r(\ell,x)$.

As a consequence, the morphism $\varphi_x$ verifies the following property: for
every $P\in\pi^{-1}(x)$, $\varphi_x^{-1}(P)\subset r(\ell,x)\ni O_x$. But in
this case, the restriction of  $\varphi_x$ to a general hyperplane
$H\subset\pi^{-1}(x)$ provides an immersion
$$H\cong\P^{n-1}\hookrightarrow\pi^{-1}(x)\cong\P^n.$$
Since $n\geq 3$, then adjunction theory tells us that this immersion is
necessarily linear and so $\varphi_x$ is linear, too.
\end{proof}

\begin{rem}\label{rem:otherproof}
{\rm As the referee suggested, the last part of the proof of the  previous lemma follows from a result of \cite{PS}.
Namely, taking a general hyperplane section $Y'$ of $Y$ such that $\overline{\varphi}(Y')$ does not contain the vertex of $Y$, and composing the restriction of $\overline{\varphi} $ to $Y'$ with the projection to $\G(1,n)$ we obtain a finite surjective morphism $Y' \simeq \G(1,n) \to \G(1,n)$, which is an isomorphism by \cite[Proposition 6]{PS}. This implies that $\varphi_x$ is  linear and concludes the proof.}
\end{rem}

%%%%%%%%%%%%%%%%%%%%%%%%%%%%%%%%%%%%%%%%%%%%%%%%%%%%%%%%%%%%%

\section{Fano bundles on Grassmannians}\label{sec:fano}

In this last section we apply our results to the study of Fano bundles (i.e.
bundles such that their projectivization is a Fano manifold) on Grassmannians;
in view of this goal,
we need to prove that, in many cases, these bundles are uniform.\\
The idea is to consider the restriction to a maximal linear subspace of the
Grassmannian;
this restriction is not in general a Fano bundle on the projective space, but
enjoys a weaker property. This motivates the following definition:

\begin{defn}\label{def:rfano} Let $(X,H)$ be a smooth complex polarized variety
and $r$ a nonnegative integer. A vector bundle $E$ on $X$ is called {\rm
$r$-Fano with respect to $H$} if $-K_{\P(E)}+r \pi^* H$ is ample, where
$\pi:\P(E) \to X$ denotes the projective bundle on $X$.
\end{defn}

Let us remark that for $r=0$ the notion of a $0$-Fano bundle is exactly that of
a Fano bundle. Let us moreover notice that $E$ is $r$-Fano if and only if the
$\Q$-vector bundle
$$E\left(\frac{-\det E-K_X+rH}{\rk{E}}\right)$$
 is ample. It is
clear from the definition that $r$-Fano implies $(r+1)$-Fano.
\medskip

Inspired by the classification of Fano bundles on projective spaces, we study
$r$-Fano bundles of rank two over the projective space. We first prove that, if
$r$
is small with respect to some function of the dimension of the base space and
of the first Chern class of the bundle $E$, then $E$ splits; then we focus on
the case
$r=1$ proving that every rank two $1$-Fano bundle on $\P^n$ splits if $n \ge
4$.\\
We follow the ideas in \cite{APW} and \cite{SW}.\par
\medskip
To simplify notations we always consider
$E$ to be normalized, that is, $c_1(E)=e$, with $e=0$ or $-1$. We will 
denote by $c_2$ the second Chern class of $E$.

Denote by $\alpha$ the class modulo $2$ of $n+1+r-e$ and let $m$ be such that
$2m = n+1+r-e +\alpha$.
With this choice of $m$ the bundle $E(m)$ is ample and the vanishing theorems of
Le Potier and Griffiths (cf. \cite[Thms.~7.3.5,~7.3.1]{L}) read as follows:

\begin{lem}\label{lem:vanish}
 With the same notation as above, we have:
\begin{itemize}
 \item (Le Potier) $ H^{i}(\P^n,E(j))=0 \text{ for } i \geq 2,  j\geq -m+r-e +\alpha$, and
 \item (Griffiths) $H^{i}(\P^n,E(j))=0 \text{ for } i \geq 1, j\geq m+r +\alpha$
\end{itemize}
\end{lem}

 Hence $E(m+r+1+\alpha)$  is
$0$-regular in the sense of Castelnuovo-Mumford and in particular:
\begin{lem}\label{lem:gg}
With the same notation as above, the vector bundle $E(m+r+1+\alpha)$ is globally generated.
\end{lem}

It is well known that not every pair $(e,c_2)$ can appear as the 
Chern classes of a vector bundle $F$. First of all, the integrality 
of the Euler characteristic can be rephrased in the following well 
known result (see \cite[pp. 112-113]{OSS}):

\begin{lem}[Schwarzenberger conditions]\label{lem:schwarz}
 Let $F$ be a rank two vector bundle on $\P^n$, and let $c_1$ and 
 $c_2$ denote its Chern classes. Then:
\begin{itemize}
 \item[$S_3$:] $c_1c_2\equiv 0$ (mod $2$)
 \item[$S_4$:] ($=S_5$) $c_2(c_2+1-3c_1-2c_1^2)\equiv 0$ (mod $12$)
 \item[$S_6$:] $-2c_2^3+254c_2^2-1644c_2\equiv 0$ (mod $720$) if 
 $c_1=-1$ and $-2c_2^3+170c_2^2-548c_2\equiv 0$ (mod $720$) if $c_1=0$.
\end{itemize}
\end{lem}

Second, the positivity of a vector bundle $F$ allows us to bound its 
second Chern class in terms of the first (cf. \cite[8.3]{L}):

\begin{lem}[Schur polynomials]\label{lem:schur}
Let $F$ be a nef rank two $\Q$-twisted vector bundle on $\P^n$, and 
let $c_1$ and $c_2$ denote its Chern classes. Then the following 
inequalities are fulfilled:
\begin{itemize}
 \item[$P_2$:] $c_2\geq 0$ if $n\geq 2$,
 \item[$P_3$:] $c_1^2\geq 2c_2$ if $n\geq 3$,
 \item[$P_4$:] 
 %$c_1^4-3c_1^2c_2+c_2^2\geq 0$ (equivalently, 
 $c_1^2\geq\dfrac{3+\sqrt{5}}{2}c_2$
 %) 
 if $n\geq 4$, and
 \item[$P_5$:] $c_1^2\geq 3c_2$ if $n\geq 5$.
\end{itemize}
If $F$ is ample, then all the inequalities are strict.
\end{lem}

We will distinguish two cases, depending on the stability of 
the vector bundle $E$. We refer the interested reader to \cite[Chapter II]{OSS} 
for details on stability of vector bundles. For our purposes 
it is enough to know the following:
\begin{lem}\label{lem:stability}
Let $F$ be a rank two vector bundle on $\P^n$, $n\geq 2$, of 
degree $c_1(E)=0$ or $-1$. The following conditions are equivalent:
\begin{itemize}
\item $F$ is stable,
\item $H^0(\P^n,F)=0$.
\end{itemize} 
Moreover, in this case the {\em discriminant} $\Delta(F):=c_1^2(F)-4c_2(F)$ is negative.
\end{lem}

Through the following two subsections we will make use 
several times of explicit computations, such as 
Riemann-Roch formula, Schur polynomials or Schwarzenberger formulas. 
We have skipped here some of these computations, that can be easily checked 
using standard computer software. 

%%%%%%%%%%%%%%%%%%%%%%%%%%%%%%%%%%%%%%%%%%%%%%%%%%%

\subsection{$1$-Fano bundle with negative discriminant}

Throughout this subsection we will assume that $\Delta(E)<0$.

Consider the following functions:
$$h_\alpha(n):=\begin{cases}\vspace{0.3cm}
\dfrac{n}{3}-\alpha  & \text{if } n=6,7, \\
\dfrac{2n-3}{3}-\alpha & \text{if } n \geq 8.
 \end{cases}
$$

\begin{prop}\label{prop:stable} Let $E$ be a rank two $r$-Fano bundle on
$\P^n$, $n\geq 6$.  If $\Delta(E)<0$ then $r\geq h_\alpha(n)$.
\end{prop}

\begin{proof} We consider the vector bundle
$F=E(m+r+1+\alpha)$ which is
globally generated by (\ref{lem:gg}). Since $F$ is globally generated
then a general section vanishes in a smooth irreducible subvariety of $\P^n$ 
of codimension two. Hence, if $r<h_\alpha(n)$, we apply 
\cite[Prop.~5.2]{APW} (which is based on
\cite[Cor.~3.4, Prop.~6.1]{HS}) to conclude that $F$ splits as a sum of 
line bundles, hence $\Delta(E)=\Delta(F)\geq 0$, a contradiction.
\end{proof}

In the particular case $r=1$ we get:
\begin{cor}\label{cor:1Fanost}
 Let $E$ be a normalized $1$-Fano bundle of rank two on $\P^n$ and degree $e$. Then:
\begin{itemize}
 \item If $n\geq 7$, then $\Delta(E)\geq 0$.
 \item If $n=6$, then $\Delta(E)\geq 0$ unless $e=-1$.
\end{itemize}
\end{cor}

In the next proposition we will study the low dimensional cases in order to get:

\begin{prop}\label{prop:rulest}
 Let $E$ be a normalized $1$-Fano bundle of rank two on 
 $\P^n$, $n\geq 4$. Then $\Delta(E)\geq 0$. 
\end{prop}
 
\begin{proof} 
Assume, on the contrary, that $\Delta(E)<0$ (equivalently $c_2>0$). We begin 
by reducing the possible set of values of $(e,c_2)$ by 
applying Lemma \ref{lem:schwarz} to $E$ and Lemma \ref{lem:schur} to $E((i_X+1-e)/2)$. A direct 
computation shows that the only possibilities are listed in the following table:
$$
\begin{array}{|c|c|c|c|}
\hline
\mbox{Case} &n&e&c_2\\\hline
\mbox{I}& 4&-1&4\\\hline
\mbox{II}&4&0&3\\\hline
\mbox{III}&5&-1&4\\\hline
\mbox{IV}&5&0&3\\\hline
%\mbox{V}&6&-1&4\\\hline
\end{array}
$$
Note that, with exception of case I, Le 
Potier vanishing (see Lemma \ref{lem:vanish}) implies that:
$$
H^i(\P^n,E(-2+j))\mbox{, for }i\geq 2, j\geq 0
$$
and in particular
$$
h^0(\P^n,E(-2+j))\geq\chi(\P^n,E(-2+j)) \mbox{, for } j\geq 0.
$$
Using Riemann-Roch formula, we obtain 
$$h^0(\P(E),\cO_{\P(E)}(1)\otimes\pi^*(\cO(-2)))=h^0(\P^n,E(-2))
\geq\chi(\P^n,E(-2))>0.$$

In every case this tells us that $E$ is unstable (cf. Lemma \ref{lem:stability}), 
but this is not enough to contradict $\Delta(E)< 0$. On the other hand, the 
effectivity of the line bundle $\cL:=\cO_{\P(E)}(1)\otimes\pi^*(\cO(-2))$ 
leads to a 
contradiction since $\cL_m:=\cO_{\P(E)}(1)\otimes\pi^*(\cO(m))$ is ample 
and one may compute that $d:=c_1(\cL)\cdot \big(c_1(\cL_m)\big)^n$ is negative. 
The possible values of this intersection number are listed in the following table:
$$
\begin{array}{|c|c|c|c|c|c|}
\hline
 \mbox{Case}&n&e&c_2&m&d\\\hline
\mbox{II}&4&0&3&3&-216\\\hline
\mbox{III}&5&-1&4&4&-1599\\\hline
\mbox{IV}&5&0&3&4&-2334\\\hline
%\mbox{V}&6&-1&4&5&-12852\\\hline
\end{array}
$$

Finally let us deal with case I. If $H^0(X,E)\neq 0$, 
an  argument similar to the one above works: taking $\cL:=\cO_{\P(E)}(1)$ 
we get $d=-171$. If $H^0(X,E)=0$, then by Lemma \ref{lem:stability} 
$E$ is stable. Note that $c_1(E(3))=5$ and $c_2(E(3))=10$, and it 
is well known (cf. \cite{DS}, see also \cite{D}) that the only 
stable bundle with these Chern classes is the 
Horrocks-Mumford bundle $\cF_{HM}$ on $\P^4$. But $\cF_{HM}$ is not $1$-Fano, 
because there are $25$ lines in $\P^4$ for which $({\cF_{HM}})_{|\P^1}
\cong\cO_{\P^1}(-1)
\oplus\cO_{\P^1}(6)$ (cf. \cite[Prop. 14]{BHM}). This concludes the proof.
\end{proof}

%%%%%%%%%%%%%%%%%%%%%%%%%%%%%%%%%%%%%%%%%%%%%%%%%%%

\subsection{$1$-Fano bundle with nonnegative discriminant}

Throughout this subsection we will assume that $\Delta(E)\geq 0$ 
(equivalently $c_2(E)\leq 0$). In particular $E$ is unstable.

Let us consider now the functions
$$
r_\alpha(n,e,c_2)=\dfrac{1}{3}\left(n-7+e-2\left\lfloor
\dfrac{e}{2}-\sqrt{n-1+\dfrac{\Delta(E)}{4}}\right\rfloor\right)-\alpha,\mbox{ for }\alpha=0,1.
$$

\begin{prop}\label{prop:unstable1} Let $E$ be a normalized rank two $r$-Fano
bundle on $\P^n$, $n\geq 2$. Let $e=c_1(E)$, $c_2=c_2(E)$ and $\alpha$ be the class modulo $2$
of $n+1+r-e$. Assume that $\Delta(E)\geq 0$ and
$r \leq r_\alpha(n,e,c_2)$.
Then $E$ splits as a sum of line bundles. 
\end{prop}

\begin{proof}
We will show that there exists $s$ such that
\begin{enumerate}
 \item $E(s)$ is globally generated, and
 \item $c_2(E(s))<(n-1)\big(c_1(E(s))-n+2\big)$,
\end{enumerate}
hence we may apply \cite[3.2]{APW} (based on
\cite[Thm.~4.2]{HS}) and conclude that $E$ splits as a sum of line bundles.

The first condition is fulfilled if $s\geq m+r+1+\alpha$, whereas
for the second we need:
$$
c_2(E)+se+s^2<(n-1)\big(e+2s-n+2\big).
$$
This inequality is equivalent to
$$
 n-1-\dfrac{e}{2}-\sqrt{n-1+\dfrac{\Delta(E)}{4}}<s<S:=
 n-1-\dfrac{e}{2}+\sqrt{n-1+\dfrac{\Delta(E)}{4}}
 $$
Note that there is always an integer in this interval. 
In fact $n\geq 2$, and $\Delta(E)\geq 0$ by hypothesis. Hence the required conditions are 
fulfilled whenever
$$
\dfrac{n+1+r-e+\alpha}{2}+r+\alpha+1=m+r+\alpha+1
\leq\lceil S \rceil-1,
$$
and this inequality is equivalent to $r \leq r_\alpha(n,e,c_2)$.
\end{proof}

Adding this result to Proposition \ref{prop:stable} we get the 
following numerical criterion for the splitting of rank two vector bundles:

\begin{cor}\label{cor:splitcrit}
Let $E$ be a normalized rank two $r$-Fano
bundle on $\P^n$, $n\geq 2$. Let $e=c_1(E)$, $c_2=c_2(E)$ and $\alpha$ be the class modulo $2$
of $n+1+r-e$. Assume that
$$
r\leq \min\left\{\left\lceil h_\alpha(n)\right\rceil-1,r_\alpha(n,e,c_2)\right\}.
$$
then $E$ splits as a sum of line bundles.
\end{cor}

Applying Proposition \ref{prop:unstable1} to $1$-Fano bundles we get:

\begin{cor}\label{cor:1Fanounst}
 Let $E$ be a normalized $1$-Fano bundle of rank two on $\P^n$, with $\Delta(E)\geq 0$. Then:
\begin{itemize}
 \item If $n\geq 7$, then $E$ splits as a sum of line bundles.
 \item If $n=6$ and $E$ does not split as a sum of line bundles, then
$c_1(E)=-1$ and $c_2\in\{-1,0\}$. 
 \item If $n=5$ and $E$ does not split as a sum of line bundles, then
$c_1(E)=0$ and $c_2\in\{-5,\dots,0\}$. 
 \item If $n=4$ and $E$ does not split as a sum of line bundles, then 
 either $c_1(E)=-1$ and $c_2\in\{-9,\dots,0\}$, or $c_1(E)=0$ and $c_2\in\{-1,0\}$.
\end{itemize}
\end{cor}

In order to rule out the cases $n=4,5,6$ we first use Lemma
\ref{lem:schwarz},
that provides severe 
restrictions on the numerical invariants of $E$:

\begin{lem}\label{lem:ruleun}
 Let $E$ be a normalized $1$-Fano bundle of rank two on $\P^n$, $n\geq 4$, with $\Delta(E)\geq 0$, and assume 
that $E$ is not isomorphic to a direct sum of line bundles. 
Then $e$ and $c_2$ take one of the following values:
$$
\begin{array}{|c|c|c|}
\hline
 \dim&e&c_2\\\hline
 4&-1&0,-2,-6\mbox{ \rm or }-8\\\hline
4&0&0\mbox{ \rm or }-1\\\hline
5&0&0,-1\mbox{ \rm or }-4\\\hline
6&-1&0\\\hline
\end{array}
$$
\end{lem}

Finally we will rule out the cases left from Corollary \ref{lem:ruleun}. 
Denote by $\beta$ the minimum integer such that $E(\beta)$ has global sections. 
Assume that $E$ is unsplit with $\Delta(E)\geq 0$. Then $E$ is unstable and 
 $\beta$ is non positive by Lemma \ref{lem:stability}. Moreover 
using duality and Le Potier vanishing (see Lemma \ref{lem:vanish}) we get:
\begin{equation*}
0>\beta>\dfrac{-n-2-e-\alpha}{2}=\left\{\begin{array}{ll}
                                     -3&\mbox{ if }n=4,\\
                                     -4&\mbox{ if }n=5,6. 
                                    \end{array}\right.
\end{equation*}
The first inequality is strict because $c_2(E)\leq 0$, 
hence $\beta$ equals $0$ if and only if $E$ splits.

Note that the vanishing of $H^0(\P^n,E(\beta-1))$ implies that $c_2(E(\beta))$ 
is either positive or zero. In the second case $E$ splits, hence we may assume that 
$c_2(E(\beta))=c_2+e\beta+\beta^2>0$. This already rules out the cases $n=4$, $e=-1$ and $c_2=-6$ or $-8$.

The rest of the cases may be excluded by showing that the numerical class 
$$\left(c_1(\cO_{\P(E)}(1))+mc_1(\pi^*(\cO_{\P^n}(1)))\right)^n$$
has negative intersection with the class of a divisor $D$ in a nonempty linear system $\left|\cO_{\P(E)}(1)(\beta')\right|$. 
Since $c_1(\cO_{\P(E)}(1))+mc_1(\pi^*(\cO_{\P^n}(1)))$ is an ample class, we get a contradiction.
In the following table we show the possible values $d$ of this intersection number for suitable choices of $\beta'\geq\beta$:

$$
\begin{array}{|c|c|c|c|c|c|}
\hline
 \dim&e&c_2&\beta&\beta'&d\\\hline
 4&-1&0&-1\mbox{ or }-2&-1&-94\\\hline
4&-1&-2&-2&-2&-187\\\hline
4&0&0&-1\mbox{ or }-2&-1&-27\\\hline
4&0&-1&-2&-2&-104\\\hline
5&0&0&-1,-2\mbox{ or }-3&-1&-256\\\hline
5&0&-1&-2\mbox{ or }-3&-2&-1198\\\hline
5&0&-4&-3&-3&-1904\\\hline
6&-1&0&-1,-2\mbox{ or }-3&-1&-7433\\\hline
\end{array}
$$

\medskip

Summing up, we have obtained the following 

\begin{prop}\label{prop:ruleunst}
 Let $E$ be a $1$-Fano bundle of rank 
 two on $\P^n$, $n\geq 4$, with $\Delta(E)\geq 0$. Then $E$ splits as a sum of line bundles.
\end{prop}

Adding this result to Proposition \ref{prop:rulest} we finally get:

\begin{thm}\label{thm:1Fano}
Let $E$ be a rank two $1$-Fano bundle on $\P^n$, $n\geq 4$.
Then $E$ splits as a sum of
line bundles.
\end{thm}

\subsection{Rank two Fano bundles on Grassmannians}

Finally we apply the results of the previous subsections, 
together with the classification of rank two uniform
vector bundles on
Grassmannians given in Theorem \ref{thm:uniformgrass},
 to the study of rank two Fano bundles on Grassmannians:

\begin{thm}\label{thm:fanograss} Let $E$ be a normalized rank two Fano bundle on
the Grassmannian $\mathbb{G}(t,n+1)$, with $n \geq 6$. Denote by $\alpha$ the
class
modulo $2$ of $n+1+t-e$. If $t<\min\{r_\alpha(n,e),h(n)\}$, then either $E$
splits as a sum of line bundles or $t=1$ and $E$ is isomorphic to the twist of
the universal bundle $\cQ(-1)$.
\end{thm}

\begin{proof} Since $E$ is Fano then the
$\Q$-vector bundle $$E\left(\frac{-\det{E}+(n+2)H}{2}\right)$$ is ample.

Denote by  $F:=E|_{\P^{n+1-t}}$ the restriction of $E$ to a linear space of
maximal dimension $\P^{n+1-t} \subset
\mathbb{G}(t,n+1)$. Just by
definition it holds that $F$ is $t$-Fano with respect to
$H=\cO_{\P^{n+1-t}}(1)$. Now we can apply Theorem \ref{thm:1Fano} and
\ref{prop:stable} to get that $F$ splits as a sum of line bundles. Since
$\rk{E}=2$ then the splitting type of the restriction of $E$ to any linear space
of maximal dimension is independent of the linear space. Hence $E$ is uniform,
being each line of $\G(t,n+1)$ contained in one of these linear spaces. We
therefore conclude by Theorem \ref{thm:uniformgrass}. \end{proof}

The same proof, in view of Theorem \ref{thm:1Fano}, shows that the
classification is complete for rank two Fano bundles on
$\mathbb{G}(1,n+1)$, $n \geq 4$:

\begin{cor}\label{cor:grass1} A Fano bundle $E$ on $\mathbb{G}(1,n+1)$, $n\geq
4$ either splits as a sum of line bundles or is isomorphic to a twist of the
universal bundle $\cQ$.
\end{cor}

\begin{ex}
{\rm The null-correlation bundle on $\P^3$ is a Fano bundle, hence a $1$-Fano bundle,
which does not split as a sum of line bundles. It has $c_1=0$ and $c_2=1$.\\
The instanton bundles with
natural cohomology on $\P^3$ constructed in \cite{H} verify that $E(2)$ is
globally
generated and hence $E(3)$ is ample, i.e, they are $1$-Fano indecomposable
bundles.
They have $c_1=0$ and $c_2=2,3$.}
\end{ex}

\begin{center}\textbf{Acknowledgements}\end{center}
We thank the referee for many useful comments and suggestions.

%%%%%%%%%%%%%%%%%%%%%%%%%%%%%%%%%%%%%%%%%%%%%%%%%%%%%%%%%%%%
\bibliographystyle{amsalpha}

\begin{thebibliography}{9999}

\bibitem[A]{A} Arrondo, E. {\it Subvarieties of Grassmannians},
  Lecture Note Series Dipartimento di Matematica Univ. Trento, 10
  (1996).

\bibitem[APW]{APW} Ancona, V., Peternell, T. and Wi\'sniewski. J.
  {\it Fano bundles and splitting theorems on projective spaces and
    quadrics}, Pacific J.  Math. {\bf 163}, no. 1 17--41 (1994).

\bibitem[AW]{AW} Andreatta, M. and Wi\'sniewski, J.  {\it On manifolds
    whose tangent bundle contains an ample locally free subsheaf},
  Invent. Math. {\bf 146}, 209--217 (2001).

\bibitem[B1]{B1} Ballico, E. {\it Uniform vector bundles on quadrics}, Ann. Univ.
Ferrara Sez. VII (N.S.) {\bf 27}, 135--146 (1982).

\bibitem[B2]{B2} Ballico, E. {\it Uniform vector bundles of rank $(n+1)$ on
$\P_n$},
Tsukuba J. Math. {\bf 7} no. 2, 215--226 (1983).

\bibitem[BHM]{BHM} Barth, W., Hulek, K. and Moore, R. {\it Shioda's 
modular surfaces S(5) and the Horrocks-Mumford bundle}, Vector bundles 
on algebraic varieties, Bombay 1984, Tata Inst. Funct. Res. Stud. 
Math. 11, 35--106 (1987).

%\bibitem[BdFL]{BdFL} Beltrametti, M.C., de Fernex, T. and Lanteri, A.
%  {\it Ample subvarieties and rationally connected fibrations}, Math.
%  Ann. {\bf 341} 897-926 (2008).

%\bibitem[BS]{BSbook} Beltrametti, M. C. and Sommese, A. J. {\it The
 %   Adjunction Theory of Complex Projective Varieties}, De Gruyter
 % Expositions in Mathematics 16, De Gruyter, Berlin-New York, 1995.

%\bibitem[BSW]{BSW} Beltrametti, M. C., Sommesse, A. J., and
 % Wi\'sniewski, J. {\it Results on varieties with many lines and their
 %   applications to adjunction theory}, Complex Algebraic Varieties
 % (Bayreuth, 1990), Lecture Notes in Mathematics 1507, Springer,
 % Berlin, 1992, pp. 16--38.

\bibitem[BG]{BG} Bloch, S. and Gieseker, D. {\it The positivity of the Chern
classes of an ample vector bundle}, Inv.
Math. {\bf 12}, 112--117 (1971).

%\bibitem[De]{De} Debarre, O {\it Higher-Dimensional Algebraic
%    Geometry}, Springer-Verlag New York, 2001.

\bibitem[D]{D} Decker, W. {\it Stable rank 2 vector bundles with 
Chern-classes $c_1=-1$, $c_2=4$}, Math. Ann. {\bf 275}, 481--500 (1986).

\bibitem[DS]{DS} Decker, W. and Schreyer, F.-O. {\it On the uniqueness 
of the Horrocks-Mumford-bundle}, Math. Ann. {\bf 273}, 415--443 (1986). 

\bibitem[Di]{Di} Dibag, I. {\it Topology of the complex varieties 
$A_s^{(n)}$}, J. Diff. Geom. {\bf 11}, no. 4, 499--521 (1976).

\bibitem[EHS]{EHS} Elencwajg, G., Hirschowithz and A., Schneider, M. {\it Les
fibr\'es uniformes de rang au plus $n$ sur $\P_n(\C)$ sont ceux qu'on croit},
Vector bundles and
differential equations (Proc. Conf., Nice, 1979), pp. 37--63,
Progr. Math., 7, Birkh\"auser, Boston, Mass., 1980.

\bibitem[E]{E} Ellia, Ph. {\it Sur les fibr\'es uniformes de rang
    $(n+1)$ sur $\P^n$}, M\'em. Soc. Math. France (N.S.) no. 7 (1982).

%\bibitem[F1]{F1} Fujita, T. {\it Vector bundles on ample divisors}, J.
%  Math. Soc. Japan {\bf 33}, no. 3, 405-414 (1981).

%\bibitem[F2]{F3} Fujita, T. {\it On Polarized manifolds whose adjoint
%    bundes are not semipositive}.  in {\it Algebraic Geometry, Sendai
%    1985.} Adv. Stud. Pure Math. {\bf 10}, 167-178 (1987).

%\bibitem[F3]{F2} Fujita, T. {\it Classification Theories of Polarized
%    Varieties}. London Mathematical Society Lecture Note Series, no.
%  155. Cambridge University Press, Cambridge, 1990.

%\bibitem[Ha]{hartshorne} Hartshorne, R. {\it Algebraic geometry}.
%  Graduate Texts in Mathematics, No. 52. Springer-Verlag, New
%  York-Heidelberg, 1977.

\bibitem[FHS]{FHS} Forster, O., Hirschowitz, A. and Schneider, M.
{\it Type de scindage g\'en\'eralis\'e pour les fibr\'es stables} in Vector
bundles and
differential equations (Proc. Conf., Nice, 1979), pp. 65--81,
Progr. Math., 7, Birkh\"auser, Boston, Mass., 1980.

\bibitem[G]{G} Guyot, M. {\it Caract\'erisation par l'uniformit\'e des fibr\'es
universels sur la Grassmannienne}. Math. Ann. {\bf 270}, 47--62 (1985).

%\bibitem[Ha]{Ha} Harris, J. {\it Algebraic geometry. A first course.} Graduate
%Texts in Mathematics, 133. Springer-Verlag, New York (1992).

\bibitem[H]{H} Hartshorne, R. {\it Stable vector bundles and instantons}, Comm.
Math. Phys. {\bf 59}, no. 1, 1--15 (1978).

\bibitem[HS]{HS} Holme and A. Schneider, M. {\it A computer aided approach to
codimension $2$ subvarieties of $\P^n$, $n\geq 6$}. J. Reine Angew. Math. {\bf
357}, 205--220 (1985).

\bibitem[Hw]{Hw} Hwang, J-M. {\it Geometry of minimal
    rational curves on Fano manifolds}, School on Vanishing Theorems
  and Effective Results in Algebraic Geometry (Trieste, 2000), ICTP
  Lect. Notes, vol. 6, Abdus Salam Int. Cent. Theoret. Phys., Trieste,
  2001, pp. 335--393.

%\bibitem[HM]{HM} Hwang, J-M and Mok, N. {\it Rigidity of irreducible
%    Hermitian symmetric spaces of the compact type under K\"ahler
%    deformation}, Invent. Math. {\bf 131}, 393-418 (1998).

%\bibitem[HM2]{HM2} Hwang, J-M and Mok, N. {\it Holomorphic maps
%  from rational homogenous spaces of Picard number one onto projective
%  manifolds}, Invent. Math. {\bf 136}, 209-231 (1999).

\bibitem[IM]{IM} Iliev A., Manivel L. {\it The Chow ring of the Cayley plane}, Compositio Math. {\bf 141}, no. 1, 146-160 (2005).

\bibitem[KS]{KS} Kachi, Y. and Sato, E. {\it Segre's reflexivity and an
inductive characterization of hyperquadrics}, Mem. Amer. Math. Soc. {\bf 160},
no. 763 (2002).

%\bibitem[KO]{KO} Kobayashi, S., Ochiai., T. {\it Characterization of
 %   complex projective spaces and hyperquadrics}, J. Math. Kyoto Univ.
 % {\bf 13}, 31-47 (1973).

\bibitem[K]{K} Koll\'ar J. {\it Rational curves on algebraic
    varieties}, Berlin, Springer-Verlag (1996).
    
\bibitem[LM]{LM} Landsberg, J.M. and Manivel, L., {\it On the 
Projective Geometry of Rational Homogeneous Varieties}, Comment. 
Math. Helv. {\bf 78}, 65--100 (2000).

\bibitem[L]{L} Lazarsfeld, R. {\it Positivity in Algebraic
    Geometry II.} Springer-Verlag, Berlin-Heidelberg, 2004.

%\bibitem[LM]{LM} Lanteri, A. and Maeda, H. {\it Ample vector bundles
%    with section vanishing on projective spaces or quadrics}, Int. J.
%  Math.{\bf 6}, no. 4, 587-600 (1995).

%\bibitem[M]{manivel} Manivel, L. {\it Gaussian maps and plethysm.
 %   Algebraic geometry (Catania, 1993/Barcelona, 1994),} Lecture Notes
  %in Pure and Appl. Math., 200, Dekker, New York, 1998, pp. 91-117.

%\bibitem[MS]{MS} Mu\~noz, R. and Sol\'a-Conde, L. E. {\it Varieties
%    swept out by Grassmannians of lines} to appear in Contemporary
%  Mathematics.

%\bibitem[NO]{nov-occ} Novelli, C., Occhetta, G. {\it Projective
%    manifolds containing a large linear subspace with nef normal
%    bundle}, preprint 2008 arxiv: 0712.3406v2.

\bibitem[PS]{PS} Paranjape, K. H. and Srinivas, V., {\it Self-maps of homogeneous spaces}, Invent. Math. {\bf 98}, 425--444 (1989).

\bibitem[S]{S} Sato, E. {\it Uniform vector bundles on a projective space}, J. Math. Soc. Japan {\bf 28}, 123--132 (1976).

\bibitem[OSS]{OSS} Okonek, C., Schneider, M. and Spindler, H. {\it
    Vector Bundles on Complex Projective Spaces}. Progress in
  Mathematics 3, Birkh\"auser, Boston, 1980.

%\bibitem[O]{O} Ottaviani, G. {\it Spinor bundles on quadrics}, Trans. of the
%Amer. Math. Soc. {\bf 307}, no. 1, 301-316 (1988).

%\bibitem[So]{sommese} Sommese, A. {\it Submanifolds of Abelian
%    varieties}, Math. Ann. {\bf 233}, no. 4, 229-256 (1978).

\bibitem[SSW]{SSW} Szurek, M. Sols, I. and Wi\'sniewski. J.
{\it Rank-$2$ Fano bundles over a smooth quadric $Q\sb 3$}.
Pacific J. Math. {\bf 148}, no. 1, 153--159 (1991).

\bibitem[SW]{SW} Szurek, M. and Wi\'sniewski. J. {\it Fano bundles
    over $\P^3$ and $Q_3$}, Pacific J. Math. {\bf 141}, no. 1 197--208
    (1990).

\bibitem[V]{V} Van de Ven, A. {\it On uniform vector bundles}, Math. Ann. {\bf
195}, 245--248 (1972).

%\bibitem[W]{Wis} Wi\'sniewski. J. {\it On Fano manifolds of large
%    index}, manuscripta math. {\bf 70}, 145-152 (1991).

\end{thebibliography}

\end{document}